\documentclass[twoside]{amsart}

\hoffset \voffset \oddsidemargin=55pt \evensidemargin=55pt
\def\deg{{\hbox{\rm deg}}}
\def\sgn{{\hbox{\rm sgn}}}
\def\id{{\hbox{\rm id}}}

\def\ZZ{{\mathbb Z}}

\def\span{{\hbox{\rm Span}}}

\def\a{\alpha}

\def\cg{{\mathcal G}}

\def\cq{{\mathcal Q}}
\def\ch{{\mathcal H}}
\def\ep{\epsilon}

\newfont{\df}{eufm10}

\def\a{\alpha}
\def\CC{{\mathbb C}}

\def\End{{\bf End}}

\title[vertex representations for $A_{2l}^{(2)}$]{\bf
Vertex representations for twisted \\ affine Lie algebra of type
$A_{2l}^{(2)}$}
\thanks{$^\star$L.M., supported by the Science Foundation of Jiangsu University
(Grant No.07JDG035).}

\author{Li-meng Xia$^\star$}
\address{LX: Faculty of Science, Jiangsu University,
Zhenjiang 212013, Jiangsu, China}

\author{Naihong Hu$^*$}
\thanks{$^*$N.H., Corresponding Author,
supported in part by the NNSF (Grants 10431040, 10728102), the
PCSIRT, the National/Shanghai Priority Academic Discipline
Programmes. \email{nhhu@math.ecnu.edu.cn}}

\address{NH: Department of Mathematics, East
China Normal University, Shanghai 200241, China}

\author{Xiaotang Bai}
\address{XB: School of Mathematical Science, Nankai University, Tianjin 300071, China}

\date{}

\begin{document}
\maketitle

\def\abstractname{ABSTRACT}
\begin{abstract}

In this paper, we construct an irreducible vertex module for twisted affine
Lie algebra of type $A_{2l}^{(2)}$.

\end{abstract}

\newtheorem{theo}{Theorem}[section]
\newtheorem{defi}[theo]{Definition}
\newtheorem{lemm}[theo]{Lemma}
\newtheorem{coro}[theo]{Corollary}
\newtheorem{prop}[theo]{Proposition}
\newtheorem{Rem}[theo]{Remark}

\setcounter{section}{-1}

\section{Introduction}
Since the first vertex construction was discovered by Lepowsky and
Wilson (1978), the vertex representations for any (untwisted) affine
Lie algebra have been constructed on certain Fock space by many
authors. Particularly, the vertex representations for
non-simply-laced cases in [XH] are given through the twisted
Heisenberg algebras. In fact, those modules are irreducible for
certain twisted affine Lie algebras. Consequently, the irreducible
vertex modules have been given for $A_{2l-1}^{(2)}, D_l^{(2)},
E_6^{(2)}, D_4^{(3)}$. However, the vertex representations for
$A_{2l}^{(2)}$ are not known yet.

In this paper, we will give an explicit construction for
$A_{2l}^{(2)}$ through its Heisenberg subalgebra. Moreover, such
vertex modules are also irreducible.

\section{Twisted affine Lie algebra of type $A_{2l}^{(2)}$}
Suppose that $\cg$ is a complex simple Lie algebra of type $A_{2l}$ and $\sigma$ its diagram antomorphism with order $2$. Let
$$\cg_i=\{x\in\cg \mid \sigma(x)=(-1)^ix\}$$
for $i\in\ZZ$. Then
$$\cg=\cg_0\oplus \cg_1$$
as $\CC$-spaces and $\cg_0$ is a simple Lie algebra of type $B_l$, $\cg_1$ is an irreducible $\cg_0$-module isomorphic to $L(\Lambda_1)$.

Let $s$ be an indeterminate, then the linear space
$$\cg^\sigma=\sum_{i\in\ZZ}\cg_i\otimes s^\frac{i}{2}\oplus \CC c\oplus \CC d$$
is an affine Lie algebra of type $A_{2l}^{(2)}$ with Lie bracket
\begin{eqnarray}
{[x\otimes s^m,y\otimes s^n]}&=&{[x,y]}\otimes s^{m+n}+\delta_{m+n,0}m(x,y)c,\\
{[d, x\otimes s^m]}&=&mx\otimes s^m,\\
{[c,\cg^\sigma]}&=&0.
\end{eqnarray}
Here $m,n\in\frac{1}{2}\ZZ$ and $(\,,)$ is a non-degenerate
invariant bilinear form on $\cg$.

Let $\ch_0$ be a Cartan subalgebra of $\cg_0$ and $\a_1,\cdots,\a_l\in\ch_0^*$ be such that
\begin{eqnarray}
(\a_i,\a_i)=\left\{\begin{array}{ll}\frac{1}{2},&i=l,\\
1,&i<l,\end{array}\right.
\end{eqnarray}
and
\begin{eqnarray}
(\a_i,\a_j)=\left\{\begin{array}{ll}-\frac{1}{2},&|i-j|=l,\\
0,&|i-j|>l,\end{array}\right.
\end{eqnarray}
so
$$\Pi=\{\a_i \mid i=1,\cdots,l\}$$
is an prime root system of $\cg_0$. Let $\cq=\span_\ZZ\{\a_1,\cdots,\a_l\}$. Note that there is a linear isomorphism
\begin{eqnarray*}
\gamma: \ch_0&\longrightarrow &\ch_0^*\\
\a_i^\vee&\longmapsto&\frac{2\a_i}{(\a_i,\a_i)}.
\end{eqnarray*}

Obviously, $\ch^\sigma=\ch_0\oplus\CC c\oplus\CC d$ is a Cartan subalgebra of $\cg^\sigma$. Extend the bilinear form of $\ch_0$ to $\ch^\sigma$ via
\begin{eqnarray}
(c,d)=1,\quad (c,\a_i)=(d,\a_i)=0,\quad i=1,\cdots,l.
\end{eqnarray}
Set $\beta=\frac{(c+d)}{\sqrt{2}}+\a_l$, then $(\beta,\beta)=\frac{3}{2}$ and $(\beta,\a_i)=(\delta_{i,l}-\delta_{i,l-1})\frac{1}{2}$.

\begin{lemm}
Suppose that $\dot\Delta$ is the root system of $\cg_0$ and
$\dot\Delta_S$ the subset of all short roots. Then $\cg^\sigma$ has
a root system
$$\Delta=\{n\delta\pm\a, (2n+1)\delta\pm2\a', m\delta \mid n\in\ZZ,
m\in\ZZ\setminus\{0\}, \a\in\dot\Delta, \a'\in\dot\Delta_S\}.$$
\end{lemm}

For convenience, define
\begin{eqnarray*}
\Delta_L&=&\{2\a \mid \a\in\dot\Delta_S\},\\
\Delta_M&=&\{\a \mid \a\in\dot\Delta\setminus\dot\Delta_S\},\\
\Delta_S&=&\{\a \mid \a\in\dot\Delta_S\}.
\end{eqnarray*}

Denoted  by $\cq_M$ the lattice generated by $\a_i \ (0<i<l)$.

\section{Vertex module}
Let $H=\CC\otimes\cq$, $H_M=\CC\otimes(\cq_M+\ZZ\beta)$ and $H(n),
H_M(n+\frac{1}{2})$ be their isomorphic copies for $n\in\ZZ$,
respectively. Then
$$\widetilde{H}=\bigoplus_{n\in\ZZ}H\left(n\right)\oplus\bigoplus_{n\in\ZZ}H_M\left(n+\frac{1}{2}\right)\oplus\CC c$$
is a Lie algebra with brackt
\begin{eqnarray}
{\left[\a'\left(m+\frac{1}{2}\right), \a''\left(n-\frac{1}{2}\right)\right]}&=&\delta_{m+n,0}\left(m+\frac{1}{2}\right)(\a',\a'')c,\\
{\left[a'\left(m\right), a''\left(n\right)\right]}&=&\delta_{m+n,0}m(\a',\a'')c,\\
{\left[\widetilde{H}, c\right]}&=&0,
\end{eqnarray}
where $\a',\a''\in H_M, a',a''\in H$, and it has a Heisenberg subalgebra
$$\widehat{H}=\bigoplus_{n\in\ZZ\setminus\{0\}}H\left(n\right)\oplus\bigoplus_{n\in\ZZ}H_M\left(n+\frac{1}{2}\right)\oplus\CC c,$$
and an abelian subalgebra
$$\widehat{H^-}=\bigoplus_{n\in\ZZ^-}H\left(n\right)\oplus\bigoplus_{n\in\ZZ^-}H_M\left(n+\frac{1}{2}\right).$$

Let $\CC[\cq]$ be the space linearly generated by $e^{\a+\lambda}
(\a\in\cq)$ and $S(\widehat{H^-})$ be the symmetric algebra
generated by $\widehat{H^-}$. Where
$$\lambda=\sum_{i=1}^l\frac{i}{2}\a_i.$$
Then we have
$$(\lambda,\a_i)=\frac{1}{4}\delta_{i,l}.$$
Define
\begin{eqnarray*}
V(\cq)=S(\widehat{H^-})\otimes \CC[\cq].
\end{eqnarray*}

\begin{theo}
$V(\cq)$ is a $\widetilde{H}$-module defined by
\begin{eqnarray}
a\left(-\frac{n}{2}\right)\cdot\left(v\otimes e^{\a+\lambda}\right)&=&a\left(-\frac{n}{2}\right)v\otimes e^{\a+\lambda},n\in\ZZ^+,\\
b\left(-n\right)\cdot\left(v\otimes e^{\a+\lambda}\right)&=&b\left(-n\right)v\otimes e^{\a+\lambda},n\in\ZZ^+,\\
b\left(0\right)\cdot\left(v\otimes e^{\a+\lambda}\right)&=&(b,\a+\lambda)v\otimes e^{\a+\lambda},\\
c\cdot\left(v\otimes e^{\a+\lambda}\right)&=&v\otimes e^{\a+\lambda},
\end{eqnarray}
and $a\left(n-\frac{1}{2}\right), b(n)(n>0)$ act as partial differential operators for which
\begin{eqnarray}
a\left(n-\frac{1}{2}\right)\cdot a'\left(m+\frac{1}{2}\right)
&=&\delta_{m+n,0}\left(n-\frac{1}{2}\right)(a,a'),\\
b\left(n\right)\cdot b'\left(m\right)&=&n\delta_{m+n,0}(a,a'),
\end{eqnarray}
for $a,\, a'\in H_M, b,b'\in H$ and $\a\in\cq$.
\end{theo}

\section{$2$-cocycle}
Define bilinear map $\epsilon: \cq\longrightarrow \{\pm1\}$ by
\begin{eqnarray}
\ep\left(\sum_{i=1}^lk_i\a_i, \sum_{j=1}^lr_j\a_j\right)
&=&\prod_{i,j=1}^l\ep(\a_i,\a_j)^{k_ir_j},\\
\ep(\a_i,\a_j)&=&\left\{\begin{array}{ll}-1,&i=j+1\\ 1,&
otherwise.
\end{array}\right.
\end{eqnarray}

\begin{lemm}
If $\a,\a',\a+\a'\in\dot\Delta$, then
\begin{eqnarray}
\ep(\a,\a')=-\ep(\a',\a).
\end{eqnarray}
\end{lemm}
\begin{proof} For the definition of this map, one can see \cite{J},
\cite{XH}. In fact it is a (sign of structure constants) $2$-cocycle
of $B_l$ type. \end{proof}

Define a map $p: \cq\longrightarrow H_M$ by
\begin{eqnarray}
&&p\left(\sum_{i=1}^lk_i\a_i\right)=\sum_{i=1}^{l-1}\sgn(k_i)\left(k_i-2\left[\frac{k_i}{2}\right]\right)\a_i+\sgn(k_l)\left(k_l-2\left[\frac{k_l}{2}\right]\right)\beta.
\end{eqnarray}
where $\sgn(k)=1$ if $k\geq0$ and $\sgn(k)=-1$ if $k<0$.
Also define
\begin{eqnarray}
p_0\left(\sum_{i=1}^lk_i\a_i\right)=\sum_{i=1}^{l}\sgn(k_i)\left(k_i-2\left[\frac{k_i}{2}\right]\right)\a_i.
\end{eqnarray}

\begin{lemm}
(1) If $\a\in\Delta_L$, then $p(\a)=p_0(\a)=0$;

(2) If $\a\in\Delta_M$, then $p(\a)=p_0(\a)=\pm(\a_i+\a_{i+1}+\cdots+\a_j)$ for some $1\leq i<j<l$;

(3) If $\a\in\Delta_S$, then
$$p(\a)=\pm\left(\a_i+\a_{i+1}+\cdots+\a_{l-1}+\beta\right)$$
 for some $1\leq i\leq l$, and $p_0(\a)=\a$.
\end{lemm}

\section{Vertex construction}
Let $z$ be a complex variable. For $\a,r\in\cq$, define define $\CC$-linear operators as
\begin{eqnarray*}
z^\a(v\otimes e^{r+\lambda})&=&z^{(\a,r+\lambda)}v\otimes e^{r+\lambda},\\
e^\a(v\otimes e^{r+\lambda})&=&v\otimes e^{\a+r+\lambda},\\
\ep_\a(v\otimes e^{r+\lambda})&=&\ep(\a,r)v\otimes e^{r+\lambda},\\
E^\pm(\a,z)&=&\exp\left(\mp\sum_{n=1}^\infty\frac{z^{\mp2n}}{n}\a(\pm n)\right),\\
F^\pm(p(\a),z)&=&\exp\left(\mp\sum_{n=1}^\infty\frac{2z^{\mp(2n-1)}}{2n-1}p(\a)\left(\pm \left(n-\frac{1}{2}\right)\right)\right),
\end{eqnarray*}
and
$$\a_i(z)=\sum_{n\in\ZZ}\a_i\left(n\right)z^{-2n}+\sum_{n\in\ZZ}p(\a_i)\left(n-\frac{1}{2}\right)z^{-2n+1},\quad i=1,\cdots,l,$$
$$\left(\sum_{i=1}^lc_i\a_i\right)(z)=\sum_{i=1}^lc_i\a_i(z).$$
Then $E^\pm(\a,z)$ and $F^\pm(\a,z)$, $\a(z)$ are elements in $\End(V(\cq))[[z,z^{-1}]]$.

Let $\widetilde{V(\cq)}$ be the formal completion of $V(\cq)$.
We give vertex operators on $\widetilde{V(\cq)}$:

1. For $\a\in\cq$, define
\begin{eqnarray}
&&Y(\a,z)=\left\{\begin{array}{ll}
E^-(\a,z)E^+(\a,z),& if\; p(\a)=0,\\
\sqrt{-1}E^-(\a,z)E^+(\a,z)F^-(p(\a),z)F^+(p(\a),z),& if\; p(\a)\in\Delta_S,\\
E^-(\a,z)E^+(\a,z)F^-(p(\a),z)F^+(p(\a),z),& if\; p(\a)\in\Delta_L.
\end{array}\right.
\end{eqnarray}

2. For $\a\in\cq$, define
\begin{eqnarray}
&&X(\a,z)=(-1)^{-p_0(\a)}Y(\a,z)z^{(\a,\a)}e^\a z^{2\a}\ep_\a.
\end{eqnarray}

Also define
\begin{eqnarray}
&&X(a,b,z,w)=(-1)^{-p_0(a+b)}:Y(a,z)Y(b,w):w^{(a+b,a+b)}e^{a+b}w^{a+b}\ep_{a+b},
\end{eqnarray}
where $:\quad:$ means the normal ordered product:
\begin{eqnarray*}
:a(m)b(n):&=&\left\{\begin{array}{ll}
a(m)b(n),&m\leq n;\\
b(n)a(m),&m>n.
\end{array}\right.
\end{eqnarray*}
for suitable $m,n\in\frac{1}{2}\ZZ$.

3. Suppose that $e_1,\cdots,e_l$ and $e_1',\cdots,e_l'$ are bases of
$H$ and $H_M$, respectively, such that
$$(e_i,e_j)=\delta_{ij},\quad (e_i',e_j')=\delta_{ij}.$$
Define operator {\small\begin{eqnarray}
&&d_0=\frac{1}{2}\sum_{i=1}^le_i(0)e_i(0)+\sum_{i=1}^l\sum_{n=1}^\infty
e_i(-n)e_i(n)+\sum_{i=1}^l\sum_{n=1}^\infty
e_i'\left(-n+\frac{1}{2}\right)e_i'\left(n-\frac{1}{2}\right).
\end{eqnarray}}
For $x=a_1(m_1)\cdots a_k(m_k)\otimes e^r\in V(\cq)$, define
$$\deg(x)=\sum_{j=1}^k m_j-\frac{1}{2}(r,r).$$
$\deg(x)$ is called the degree of $x$. Then
$$d_0(x)=-\deg(x).$$

\begin{lemm}
\begin{eqnarray}
{[d_0, \a(m)]}&=&-m\a(m),\\
{[d_0, X(\a,z)]}&=&\frac{1}{2}\frac{z\partial}{\partial z}X(\a,z).
\end{eqnarray}
\end{lemm}

\begin{proof} It is clear that
$$\deg(\a(m)\cdot (v\otimes e^r))=(m+\deg(v\otimes e^r))(\a(m)\cdot (v\otimes e^r)),$$
so
$$[d_0,\a(m)]\cdot(v\otimes e^r)=-m\a(m)\cdot (v\otimes e^r),$$
then (25) holds.
From (25), we have
\begin{equation*}
\begin{split}
\Bigl[d_0,X(\a,z)\Bigr]&=\left(\sum_{n=1}^\infty \a(-n)z^{2n}
+p(\a)\left(-n{+}\frac{1}{2}\right)z^{2n-1}+\a(0)+\frac{(\a,\a)}{2}\right)X(\a,z)\\
&\quad+\,X(\a,z)\left(\sum_{n=1}^\infty
\a(n)z^{-2n}+p(\a)\left(n-\frac{1}{2}\right)z^{-2n+1}\right),
\end{split}
\end{equation*}
which equals
$$\frac{1}{2}\frac{z\partial}{\partial z}X(\a,z).$$
\end{proof}

\begin{lemm}
For any $\a\in\Delta_L\cup\Delta_M\cup\Delta_S$, the Laurent series of $X(\a,z)$ are denoted by
$$X(\a,z)=\sum_{n\in\ZZ}X_\frac{n}{2}(\a)z^{-n}.$$
Particularly, if $\a\in\Delta_L$, we have
$$X(\a,z)=\sum_{n\in\ZZ}X_{n+\frac{1}{2}}(\a)z^{-n}.$$
\end{lemm}
\begin{proof} Let $r\in\cq$. If $\a\in\Delta_L$, then $\a=\pm2(\a_i+\cdots\a_l)$ for
some $i>0$, and $(\a,\lambda)=\pm\frac{1}{2}$, for any
$r\in\cq$, $(\a,r)\in\ZZ$, so
\begin{eqnarray*}
{2[\deg(e^\a\cdot e^{\lambda+r})-\deg(e^{\lambda+r})]}&=&
(\a+r+\lambda,\a+r+\lambda)-(r+\lambda,r+\lambda)\\
&=&(\a,\a)+2(\a,r+\lambda)\\
&=&2\pm1+2(\a,r)\in2\ZZ+1,
\end{eqnarray*}
If  $\a\in\Delta_S$, then $\a=\pm(\a_i+\cdots\a_l)$ for some $i>0$,
\begin{eqnarray*}
{2[\deg(e^\a\cdot e^{\lambda+r})-\deg(e^{\lambda+r})]}&=&
(\a+r+\lambda,\a+r+\lambda)-(r+\lambda,r+\lambda)\\
&=&(\a,\a)+2(\a,r+\lambda)\\
&=&\frac{1}{2}\pm\frac{1}{2}+2(\a,r)\in\ZZ,
\end{eqnarray*}
If  $\a\in\Delta_M$, then
$\a=\pm((\a_i+\cdots+\a_{j-1}+\a_j+\cdots\a_l)\pm(\a_j+\cdots+\a_l))$
for some $l\geq j>i>0$, so $(\a,\lambda)=1$ or $0$, thus
\begin{eqnarray*}
{2[\deg(e^\a\cdot e^{\lambda+r})-\deg(e^{\lambda+r})]}&=&
(\a+r+\lambda,\a+r+\lambda)-(r+\lambda,r+\lambda)\\
&=&(\a,\a)+2(\a,r+\lambda)\\
&=&1+2(\a,r)\pm1\in\ZZ\\
or&=&1+2(\a,r)\in\ZZ
\end{eqnarray*}
Then by Equation(25) and the definition of $E,\,F$ operators, we
know that the lemma holds. \end{proof}
\begin{theo}
The Lie algebra linearly generated by operators
$$\left\{\left.X_\frac{n}{2}(\a),X_n(\a'),e_i(n),e_i'\left(n{+}\frac{1}{2}\right),\id,
 d_0 \right| \a\in\Delta_M\cup\Delta_S, \a'\in\Delta_L, n\in\ZZ, i=1,\cdots,l\right\}$$
on $V(\cq)$ is isomorphic to the twisted affine Lie algebra of type
$A_{2l}^{(2)}$. The isomorphism $\pi$ is given by
\begin{eqnarray*}
\pi(e_\a\otimes s^\frac{n}{2})&=&X_\frac{n}{2}(\a),\quad \a\in\Delta_M\cup\Delta_S,\\
\pi(e_\a\otimes s^{n+\frac{1}{2}})&=&X_{n+\frac{1}{2}}(\a),\quad \a\in\Delta_L,\\
\pi(\gamma^{-1}(\a_i)\otimes s^n)&=&\a_i(n),\\
\pi(c)&=&\id,\\
\pi(d)&=&-d_0.
\end{eqnarray*}
Additionally, the image of imaginary root vectors with non-integer
degree can be obtained from the above definition, together with Lie
bracket.
\end{theo}

\section{Proof of Theorem 4.3}
\begin{lemm}
$$[a(n),X_\frac{m}{2}(\a)]=(a,\a)X_{n+\frac{m}{2}}(\a),\quad m,n\in\ZZ.$$
\end{lemm}

By analogy of the argument in Section 3.4 of [FLM], we easily obtain
by a direct calculation
\begin{lemm}
Suppose that $z$ and $w$ are
two complex variables. Then
$$E^+(a, z)E^-(b, w)
=z^{-2(a, b)}(z^2-w^2)^{(a, b)}E^-(b,w)E^+(a, z), $$
$$F^+(p(a), z)F^-(p(b), w)
=(z-w)^{(p(a) , p(b))} (z+w)^{-(p(a), p(b))}F^-(p(b), w)F^+(p(a), z),$$
for $|z|>|w|$ and $a,\; b\in\Delta_L\cup\Delta_M\cup\Delta_S\cup\{0\}$.
\end{lemm}

For complex variables $z,w$, in this paper, $C_1$ means the field
such that $|z|>|w|$ and $C_2$ means the field such that $|z|<|w|$.

\begin{lemm}
For $\a\in\Delta_S$,
$$[X_\frac{m}{2}(\a), X_\frac{n}{2}(-\a)]=-2\ep(\a,-\a)(\delta_{m+n,0}m+2\a(z)).$$
\end{lemm}
\begin{proof}
\begin{eqnarray*}
&&{[X_\frac{m}{2}(\a), X(-\a,w)]}\\
&&=-\frac{1}{2\pi\sqrt{-1}}\int_{C_1-C_2}
\ep(a,-a)z^{m-1}\frac{z(z+w)}{(z-w)^2}X(a,-a,z,w)(zw^{-1})^{2a-\frac{1}{2}}dz\\
&&=-2\ep(a,-a)w^m(m+2\a(w)),
\end{eqnarray*}
so it is true. \end{proof}

\begin{lemm}
For $\a\in\Delta_M$,
$$[X_\frac{m}{2}(\a), X_\frac{n}{2}(-\a)]=\ep(\a,-\a)(\delta_{m+n,0}m+2\a(z)).$$
\end{lemm}
\begin{proof}\begin{eqnarray*}
&&{[X_\frac{m}{2}(\a), X(-\a,w)]}\\
&&=\frac{1}{2\pi\sqrt{-1}}\int_{C_1-C_2}
\ep(a,-a)z^{m-1}\frac{z^2}{(z-w)^2}X(a,-a,z,w)(zw^{-1})^{2a-1}dz\\
&&=\ep(a,-a)w^m(m+2\a(w)),
\end{eqnarray*}
that is, the lemma holds. \end{proof}

\begin{lemm}
For $\a\in\Delta_L$,
$$[X_\frac{m}{2}(\a), X_\frac{n}{2}(-\a)]=\frac{1}{2}\ep(\a,-\a)(\delta_{m+n,0}m+2\a(z)).$$
\end{lemm}
\begin{proof}\begin{eqnarray*}
&&{[X_\frac{m}{2}(\a), X(-\a,w)]}\\
&&=\frac{1}{2\pi\sqrt{-1}}\int_{C_1-C_2}
\ep(a,-a)z^{m-1}\frac{z^4}{(z-w)^2(z+w)^2}X(a,-a,z,w)(zw^{-1})^{2a-1}dz\\
&&=\ep(a,-a)w^m(m+2\a(w)),
\end{eqnarray*}
hence, the lemma holds. \end{proof}

\begin{lemm}
 If $a,b\in\Delta_S$ and $a+b\in\Delta_M$, then
$$[X_\frac{m}{2}(a), X_\frac{n}{2}(b)]=-2\ep(a,b)X_\frac{m+n}{2}(a+b),$$
when $p_0(a+b)=a+b$ and
$$[X_\frac{m}{2}(a), X_\frac{n}{2}(b)]=2\sqrt{-1}\ep(a,b)X_\frac{m+n}{2}(a+b),$$
when $p_0(a+b)=a+b$.
\end{lemm}
\begin{proof}(A) if $p_0(a+b)=a+b$, then $(a,b)=0, (p(a),p(b))=-1$, so
\begin{eqnarray*}
&&{[X_\frac{m}{2}(a), X(b,w)]}\\
&&=-\frac{1}{2\pi\sqrt{-1}}\int_{C_1-C_2}\ep(a,b)z^{m-1}\frac{z+w}{z-w}X(a,b,z,w)(zw^{-1})^{2a+\frac{1}{2}}dz\\
&&=-2\ep(a,b)w^mX(a+b,w).
\end{eqnarray*}
(B) if not, we can assume that $p_0(a+b)=b-a$, then $(p(a),p(b))=1$, so
\begin{eqnarray*}
&&{[X_\frac{m}{2}(a), X(b,w)]}\\
&&=-\frac{1}{2\pi\sqrt{-1}}\int_{C_1-C_2}\ep(a,b)z^{m-1}(-1)^{-2a}\frac{z-w}{z+w}X(a,b,z,w)(zw^{-1})^{2a+\frac{1}{2}}dz\\
&&=2\sqrt{-1}\ep(a,b)w^mX(a+b,w),
\end{eqnarray*}
hence, this lemma holds. \end{proof}

\begin{lemm}
 If $a,b\in\Delta_S$ and $a+b\in\Delta_L$, then $a=b$ and
$$[X_\frac{m}{2}(a), X_\frac{n}{2}(a)]=(-1)^{m}4\sqrt{-1}\ep(a,a)X_\frac{m+n}{2}(2a).$$
\end{lemm}
\begin{proof} Now
\begin{eqnarray*}
&&{[X_\frac{m}{2}(a), X(a,w)]}\\
&&=-\frac{1}{2\pi\sqrt{-1}}\int_{C_1-C_2}
\ep(a,b)z^{m-1}(-1)^{2a}\frac{(z-w)^2}{z(z+w)}X(a,b,z,w)(zw^{-1})^{2a+\frac{3}{2}}dz\\
&&=(-1)^m4\sqrt{-1}\ep(a,b)w^mX(2a,w),
\end{eqnarray*}
so this lemma is true. Note that the coefficient of bracket is not zero
if and only $m+n\in2\ZZ+1$, this coincides with that
$$X(2a,z)=\sum_{n\in\ZZ}X_{n+\frac{1}{2}}(2a)z^{-2n-1}.$$
\end{proof}

\begin{lemm}
 If $b\in\Delta_M$ and $a,\, a+b\in\Delta_S$, then
$$[X_\frac{m}{2}(a), X_\frac{n}{2}(b)]=\ep(a,b)X_\frac{m+n}{2}(a+b).$$
\end{lemm}
\begin{proof}In this case, $p_0(a+b)=p_0(a)+p_0(b)$ and $(a,b)=(p(a),p(b))=-\frac{1}{2}$, so
\begin{eqnarray*}
&&{[X_\frac{m}{2}(a), X(a,w)]}\\
&&=\frac{1}{2\pi\sqrt{-1}}\int_{C_1-C_2}\ep(a,b)z^{m-1}\frac{z}{(z-w)}X(a,b,z,w)
(zw^{-1})^{2a-\frac{1}{2}}dz\\
&&=\ep(a,b)w^mX(a+b,w),
\end{eqnarray*}
hence, the lemma is true.
\end{proof}

\begin{lemm}
 If $a,b,a+b\in\Delta_M$, then
$$[X_\frac{m}{2}(a), X_\frac{n}{2}(b)]=(-1)^{2m}X_\frac{m+n}{2}(a+b).$$
\end{lemm}
\begin{proof}(A) if $p_0(a+b)=p_0(a)+p_0(b)$, then $(a,b)=(p(a),p(b))=-\frac{1}{2}$, so
\begin{eqnarray*}
&&{[X_\frac{m}{2}(a), X(b,w)]}\\
&&=\frac{1}{2\pi\sqrt{-1}}\int_{C_1-C_2}\ep(a,b)z^{m-1}\frac{z}{z-w}X(a,b,z,w)(zw^{-1})^{2a}dz\\
&&=\ep(a,b)w^mX(a+b,w).
\end{eqnarray*}
(B) if not, we can assume that $p_0(a+b)=p_0(b)-p_0(a)$, then $(a,b)=-\frac{1}{2}$ and $(p(a),p(b))=\frac{1}{2}$, so
\begin{eqnarray*}
&&{[X_\frac{m}{2}(a), X(b,w)]}\\
&&=\frac{1}{2\pi\sqrt{-1}}\int_{C_1-C_2}\ep(a,b)(-1)^{-2p_0(a)}z^{m-1}(-1)^{-2a}\frac{z-w}{z+w}X(a,b,z,w)(zw^{-1})^{2a}dz\\
&&=(-1)^{-2p_0(a)+2a}\ep(a,b)w^mX(a+b,w),
\end{eqnarray*}
it is clear that $2(a-p_0(a))\in4\dot\Delta\cup\{0\}$, so
$(-1)^{-2p_0(a)+2a}=\id_{V(\cq)}$, hence, this lemma holds.
\end{proof}

\begin{lemm}
 If $a,b\in\Delta_M$ and $a+b\in\Delta_L$, then
$$[X_\frac{m}{2}(a), X_\frac{n}{2}(b)]=2\ep(a,b)(-1)^{m}X_\frac{m+n}{2}(a+b),$$
if $(p(a),p(b))=1$ and
$$[X_\frac{m}{2}(a), X_\frac{n}{2}(b)]=2\ep(a,b)X_\frac{m+n}{2}(a+b),$$
if $(p(a),p(b))=-1$.
\end{lemm}
\begin{proof} In this case, there must be $p_0(a)=\pm p_0(b)$,
(A) if $p_0(a)=p_0(b)$, then $(a,b)=0$ and $(p(a),p(b))=1$, so
\begin{eqnarray*}
&&{[X_\frac{m}{2}(a), X(b,w)]}\\
&&=\frac{1}{2\pi\sqrt{-1}}\int_{C_1-C_2}\ep(a,b)z^{m-1}(-1)^{-2p_0(a)}\frac{z-w}{z+w}X(a,b,z,w)(zw^{-1})^{2a}dz\\
&&=2(-1)^m\ep(a,b)w^mX(a+b,w)(-1)^{2(a-p_0(a))}.
\end{eqnarray*}
(B) if $p_0(a)=-p_0(b)$, then $(a,b)=0$ and $(p(a),p(b))=-1$, so
\begin{eqnarray*}
&&{[X_\frac{m}{2}(a), X(b,w)]}\\
&&=\frac{1}{2\pi\sqrt{-1}}\int_{C_1-C_2}\ep(a,b)z^{m-1}\frac{z+w}{z-w}X(a,b,z,w)(zw^{-1})^{2a}dz\\
&&=2\ep(a,b)w^mX(a+b,w).
\end{eqnarray*}
\end{proof}

\begin{lemm}
 If $b,\, a+b\in\Delta_M$ and $a\in\Delta_L$, then
$$[X_\frac{m}{2}(a), X_\frac{n}{2}(b)]=(-1)^{m}X_\frac{m+n}{2}(a+b),$$
when $p_0(a)=p_0(a+b)$ and
$$[X_\frac{m}{2}(a), X_\frac{n}{2}(b)]=(-1)^{m}\frac{1+(-1)^n}{2}X_\frac{m+n}{2}(a+b),$$
Here, $m$ is odd number.
\end{lemm}
\begin{proof} At first, we know that $(a,b)=-1$, then if $p_0(b)=p_0(a+b)$, we have
\begin{eqnarray*}
&&{[X_\frac{m}{2}(a), X(b,w)]}\\
&&=\frac{1}{2\pi\sqrt{-1}}\int_{C_1-C_2}\ep(a,b)z^{m-1}\frac{z^2}{(z-w)(z+w)}X(a,b,z,w)(zw^{-1})^{2a}dz\\
&&=\ep(a,b)w^mX(a+b,w),
\end{eqnarray*}
this equation hold since $(-1)^{2a}=-\id_{V(\cq)}$.

If $p_0(b)=-p_0(a+b)$, by Lemmas 5.1, 5.4 and 5.10, we have
\begin{eqnarray*}
&&{[X_\frac{m}{2}(a), X(b,w)]}\\
&&=\frac{1}{2}(-1)^m\ep(-b,a+b){[[X_\frac{m}{2}(-b), X_0(a+b)], X_\frac{n}{2}(b)]}\\
&&=\frac{1}{2}(-1)^m\ep(-b,a+b)\ep(-b,b)\left[w^m(m+2b(w)),X_0(a+b)\right]\\
\end{eqnarray*}
since $(b,a+b)=0$, $(p(b),p(a+b))=-1$ and $m$ is odd, so if $n$ is even, we have
$$[X_\frac{m}{2}(a), X_\frac{n}{2}(b)]=(-1)^{2m}X_\frac{m+n}{2}(a+b),$$
otherwise, it is zero. Additionally, $\ep(a,b)=1$ since
$a\in2\cq$. Thus we have obtained the result.
\end{proof}

\begin{lemm}
If $b,a+b\in\Delta_S$ and $a\in\Delta_L$, then $a+2b=0$,
$$[X_\frac{m}{2}(-2b), X_\frac{n}{2}(b)]=(-1)^{m}\sqrt{-1}X_\frac{m+n}{2}(-b),$$
\end{lemm}
\begin{proof}
\begin{eqnarray*}
&&{[X_\frac{m}{2}(-2b), X(b,w)]}\\
&&=\frac{1}{2\pi\sqrt{-1}}\int_{C_1-C_2}\ep(-2b,b)z^{m-1}
\frac{z^2}{(z-w)(z+w)}X(a,b,z,w)(zw^{-1})^{-4b}dz\\
&&=\sqrt{-1}\ep(-2b,b)w^mX(a+b,-w),
\end{eqnarray*}
so this lemma holds since $\ep(-2b,b)=1$.
\end{proof}

By all these lemmas and Lemma 4.1, we know that Theorem 4.3 is true.

\section{The structure of $V(\cq)$}
Let
$$\a_0^\vee=c-\sum_{i=1}^{l-1}\a_i^\vee-\frac{1}{2}\a_l^\vee\in\ch^\sigma.$$
Choose $\a_0\in\ch^{\sigma *}$ such that $\{\a_0,\a_1,\cdots,\a_l\}$
is the simple root system of twisted affine Lie algebra $\cg^\sigma$ and
$$\a_0(d)=1,\quad \a_0(\a_0^\vee)=2,\quad \a_0(c)=0$$
and
$$\a_0(\a_i^\vee)=-\delta_{i,1}$$
for $i=1,\cdots,l$. Then $\delta=\a_0+2(\a_1+\cdots+\a_l)$ is an
imaginary root of $\cg^\sigma$. Let $\Lambda_i\in\ch^{\sigma *}$ be
such that $\Lambda_i(\a_j^\vee)=\delta_{i,j}$ for $i=0,1,\cdots,l$.

\begin{lemm}
$V(\cq)$ is a completely reducible module and associated with Cartan
subalgebra $\ch^\sigma$, it has weight space decompostion
$$V(\cq)=\sum_{\mu\in P(V(\cq))}V(\cq)_\mu.$$
\end{lemm}
The proof is very similar to those in \cite{J},\cite{LH} and
\cite{XH}.

\begin{lemm}
If $x$ is a highest weight vector, then it must have the form $1\otimes e^{\lambda+\a}$.
\end{lemm}
\begin{proof} It is clear since $x$ must be commutative
with $\a_i(n)$ and $p(\a_i)(n+\frac{1}{2})$ for any $n\in\ZZ$ and
$i=1,\cdots,l$.
\end{proof}

\begin{lemm}
$1\otimes e^\lambda$ is a highest weight vector.
\end{lemm}
\begin{proof} Obviously, for any $0<i<l$, we have
$$X_0(\a_i)\cdot(1\otimes e^\lambda)=Y_\frac{1}{2}(\a_i)\otimes e^{\lambda+\a_i}=0,$$
and
$$X_0(\a_l)\cdot(1\otimes e^\lambda)=Y_\frac{1}{2}(\a_i)\otimes e^{\lambda+\a_l}=0.$$
Finally,
$$X_1(-2\a_1-\cdot-2\a_l)\cdot(1\otimes e^\lambda)=Y_1(-2\a_1-\cdots-2\a_l)
\otimes e^{\lambda-2\a_1-\cdots-2\a_l}=0.$$ So $1\otimes e^\lambda$
is a highest weight vector. Particularly, by a direct computation,
we know that the irreducible submodule with highest weight vector
$1\otimes e^\lambda$ has highest weight $\Lambda_l$.
\end{proof}

\begin{theo}
$V(\cq)$ is an irreducible $\cg^\sigma$-module isomorphic to
$L(\Lambda_l)$.
\end{theo}
\begin{proof} If $x=1\otimes e^{\lambda+\a}$ is a highest weight vector,
then for $i=1,\cdots,l$, by
$$X_0(\a_i)\cdot(x)=0,$$
we have
$$(\a,\a_i)>-\frac{1}{2},$$
that is to say,
$$(\a,\a_i)\geq 0.$$
Secondly, by
$$X_\frac{1}{2}(-2\a_1-\cdots-2\a_l)\cdot(x)=0,$$
we have
$$(\a,\a_1+\cdots+\a_l)\leq \frac{1}{4}.$$
Then $\a=0$. Thus we have proved that $1\otimes e^\lambda$ is the
unique highest weight vector (up to a scalar). That's, $V(\cq)$ is
irreducible.
\end{proof}

\vskip30pt
\def\refname{

\end{document}